\documentclass{amsart}

\usepackage{amsmath}
\usepackage{amsfonts}
\usepackage{amssymb}
\usepackage{amsthm}
\usepackage{graphics}
\usepackage{color}
\usepackage{algorithm}
\usepackage{algorithmic}

\input xy 
\xyoption{all}

\newtheorem{theorem}{Theorem}[section]
\newtheorem{lemma}[theorem]{Lemma}

\newtheorem{proposition}[theorem]{Proposition}

\theoremstyle{definition}
\newtheorem{definition}[theorem]{Definition}

\theoremstyle{remark}
\newtheorem{remark}[theorem]{Remark}

\numberwithin{equation}{section}

\newcommand{\qform}[1]{{\left\langle{#1}\right\rangle}}

\DeclareMathOperator{\End}{End}

\DeclareMathOperator{\Nrd}{Nrd}

\DeclareMathOperator{\res}{res}

\DeclareMathOperator{\cor}{cor}

\DeclareMathOperator{\point}{point}

\DeclareMathOperator{\et}{et}

\DeclareMathOperator{\fppf}{fppf}

\begin{document}

\title{Zero Cycles of Degree One on Principal Homogeneous Spaces}

\author{Jodi Black}

\address{Department of Mathematics and Computer Science, Emory University, Atlanta, Georgia, 30322}

\email{jablack@learnlink.emory.edu}

\thanks{The results in this work are from a doctoral dissertation in progress under the direction of R. Parimala. The author wishes to thank her for her encouragement and support without which this work would certainly not have been possible. Thanks also to Skip Garibaldi and Jean-Pierre Tignol for providing thoughtful comments on earlier drafts of this work which improved it considerably. Finally, the author wishes to thank Eva Bayer and EPFL for hospitality provided during the later stages of the preparation of this work.}

\date{July 1, 2010}

\begin{abstract}
Let $k$ be a field of characteristic different from 2. Let $G$ be a simply connected or adjoint semisimple algebraic $k$-group which does not contain a simple factor of type $E_8$ and such that every exceptional simple factor of type other than $G_2$ is quasisplit. We show that if a principal homogeneous space under $G$ over $k$ admits a zero cycle of degree 1 then it has a $k$-rational point.

\end{abstract}

\maketitle

\section*{Introduction}

Let $k$ be a field and let $G$ be an absolutely simple algebraic group defined over $k$. Let $S(G)$ be the set of homological torsion primes of $G$ defined by Serre \cite{Serre}. 

\begin{definition}
We say that a number $d$ is \emph{coprime to $S(G)$} if none of its prime factors is contained in $S(G)$. 
\end{definition}

The following question of Serre \cite[pg 233]{Serre} is open in general.

\bigskip

\begin{description}
\item[Q] Let $k$ be a field and let $G$ be an absolutely simple $k$-group. Let $\{L_i\}_{1 \leq i \leq m}$ be a set of finite field extensions of $k$ and let the greatest common divisor of the degrees of the extensions $[L_i:k]$ be d. If $d$ is coprime to $S(G)$, does the canonical map 
\[
H^1(k,G) \to \displaystyle \prod_{i=1}^m H^1(L_i,G)
\]
have trivial kernel?
\end{description} 

\bigskip

The above question has great implications. For instance, a positive answer for exceptional groups would lead to the solution of Serre's Conjecture II  \cite[Chapter III, \S 3.1]{SerreGC} for these groups, which is still open. Zinovy Reichstein \cite[Section 5]{Zinovy} has distinguished between \emph{Type 1} and \emph{Type 2} problems in Galois Cohomology. The former type can be conveniently handled with current methods while the latter poses greater difficulties. A positive answer to {\bf Q} would reduce the Type 2 problem of finding points on principal homogeneous spaces over a general field to the Type 1 problem of finding points over fields with absolute Galois group a pro-$p$ group. 

 The first major result in this direction for a general field $k$ is due to Bayer-Lenstra \cite[Section 2]{BayerLenstra} for groups of isometries of algebras with involution. Our approach in this paper is to build on the theorem of Bayer-Lenstra to prove the following:

\begin{theorem} \label{intrometa}
Let $k$ be a field of characteristic different from 2. Let $\{L_i\}_{1 \leq i \leq m}$ be a set of finite field extensions of $k$ and let $\gcd([L_i:k])=d$. Let $G$ be a simply connected or adjoint semisimple algebraic $k$-group which does not contain a simple factor of type $E_8$ and such that every exceptional simple factor of type other than $G_2$ is quasisplit. If $d$ is coprime to $S(G)$, then the canonical map
\[
H^1(k,G) \to \prod_{i=1}^m H^1(L_i,G)
\]
has trivial kernel.
\end{theorem}

Notable consequences of this result include the following:

\begin{theorem} \label{PHS}
Let $k$ be a field of characteristic different from 2 and let $G$ be a simply connected or adjoint semisimple algebraic $k$-group which does not contain a simple factor of type $E_8$ and such that every exceptional simple factor of type other than $G_2$ is quasisplit.. Let $X$ be a principal homogeneous space under $G$ over $k$. If $X$ admits a zero cycle of degree one then $X$ has a rational point.  
\end{theorem}

This is just the case $d=1$ of \ref{intrometa} and gives a positive answer to a question posed by Serre \cite[pg 192]{SerreGC} for these groups. We remark that \ref{PHS} does not hold if $X$ is instead a quasi-projective homogeneous variety \cite{Florence} or a projective homogeneous variety \cite{Parihomog}.

Theorem \ref{intrometa} implies that the Rost invariant $R_G:H^1(k,G) \to H^3(k,\mathbb Z/3\mathbb Z)$ is injective for $G$ an absolutely simple, simply connected group of type $A_2$. In turn, one recovers a classification of unitary involutions on algebras of degree 3 \cite[Theorem 19.6]{KMRT} which had appeared in \cite{HKRT}.

Triviality of the kernel of the Rost invariant on torsors under a quasisplit, simply connected group of type trialitarian $D_4$, $E_6$ and $E_7$ \cite{Skip}, leads immediately to the proof of \ref{intrometa} in these cases. Springer's theorem \cite[Chapter VII, Theorem 2.7]{Lam} applied to norm forms of Cayley algebras and trace forms of Jordan algebras is crucial for the proof of \ref{intrometa} for a simple group of type $G_2$ and a split, simple group of type $F_4$ respectively. Also important for the $F_4$ case is a mod 3 invariant due to Rost \cite{Rostmod3}.

Besides the Bayer-Lenstra theorem \cite{BayerLenstra}, the Gille-Merkurjev norm principle theorems \cite[Th\'eor\`eme II.3.2]{GilleNP}, \cite[Theorem 3.9]{MerkurjevNP} are critical for our proof for simply connected classical groups. In the case of adjoint classical groups, we make regular use of Bayer-Lenstra's \cite{BayerLenstra} extension of Scharlau's transfer homomorphism to hermitian forms. 

As we wish to exploit Weil's classification of classical semisimple algebraic groups as groups associated to algebras with involution, we begin by recalling relevant notions from the theory of algebras with involution. In section \ref{SG} we give the classification of simply connected and adjoint absolutely simple groups and give values of $S(G)$ for each group $G$. The final preliminary section is section \ref{GC} where we recall the basic notions from Galois Cohomology which we will need in the remainder of the paper. In Section \ref{SC} we consider the question {\bf Q} for $G$ a simply connected, absolutely simple, classical group and in section \ref{AD} for $G$ an adjoint, absolutely simple, classical group. Section \ref{EX} is a discussion of the question {\bf Q} for $G$ quasisplit, simple exceptional of type other than $E_8$. Finally in Section \ref{ME} we show that our main result  \ref{intrometa} is an easy consequence of our results in the preceding three sections. 

\section{Algebras with Involution} 

Let $A$ be a central simple algebra over a field $K$ of characteristic different from 2. An \emph{involution} $\sigma$ on $A$ is an anti-automorphism of period 2. We will often write $(A, \sigma)$ for a central simple K-algebra $A$ with involution $\sigma$. Let $k$ denote the set of elements in $K$ fixed by $\sigma$. If $k = K$ we call $\sigma$ an \emph{involution of the first kind}. An involution of the first kind is called \emph{orthogonal} if it is a form for the transpose over $\bar k$ and \emph{symplectic} otherwise. If $k \neq K$ we call $\sigma$ an \emph{involution of the second kind}. In this case, $[K:k]=2$. An involution of the second kind will also be referred to as an \emph{involution of  unitary type}. 

A  \emph{similitude} of a central simple algebra with involution $(A,\sigma)$ is an element $a \in A$ such that $\sigma(a)a$ is in $k^*$. This element $\sigma(a)a$ is called the \emph{multiplier} of $a$ written $\mu(a)$. Following \cite{KMRT} we denote the group of similitudes of $(A, \sigma)$ by $GO(A, \sigma)$ if $\sigma$ is of orthogonal type, $GSp(A, \sigma)$ if $\sigma$ is of symplectic type and $GU(A, \sigma)$ if $\sigma$ is of unitary type. Let the quotients of these groups by their centers be denoted by $PGO(A, \sigma)$, $PGSp(A, \sigma)$ and $PGU(A, \sigma)$ respectively, and let them be referred to as the group of \emph{projective similitudes} of $(A, \sigma)$ in each case. The group of similitudes with multiplier 1 is called the group of \emph{isometries} of $(A, \sigma)$ and is denoted $O(A, \sigma)$, $Sp(A,\sigma)$ and $U(A, \sigma)$ in the cases $\sigma$ orthogonal, symplectic and unitary respectively. 

Let $SU(A, \sigma)$ be the elements in $U(A, \sigma)$ with trivial reduced norm. For $\sigma$ an orthogonal involution on a central simple $K$-algebra $A$ of even degree, let $GO^+(A, \sigma)$ denote the set of elements $a$ in $GO(A, \sigma)$ such that $\Nrd(a) = \mu(a)^{\text{deg}(A)/2}$ and $PGO^+(A, \sigma)$ be the quotient of $GO^+(A, \sigma)$ by its center. Let $GO^-(A, \sigma)$ be the coset of $GO^+(A, \sigma)$ in $GO(A, \sigma)$ consisting of elements $a$ such that $\Nrd(a) = - \mu(a)^{\text{deg}(A)/2}$. We will call elements of $GO^+(A, \sigma)$ \emph{proper similitudes} and those of $GO^-(A, \sigma)$ \emph{improper similitudes}. For $(A, \sigma)$ an algebra of even degree with orthogonal involution, let $Spin(A, \sigma)$ be the subgroup of the Clifford group consisting of elements $g$ with $g \tilde \sigma(g) =1$ where $\tilde \sigma$ is the map on the Clifford group induced by $\sigma$. We also recall that for a $K$-algebra $A$, $SL_1(A)$ is the kernel of the reduced norm map on $GL_1(A)$ and $PGL_1(A)$ is the quotient of $GL_1(A)$ by its center. 

Given a central simple $K$-algebra $A$ with involution $\sigma$ and $K^{\sigma}=k$, a \emph{hermitian form} $h$ on a right $A$-module $V$  is a map $h:V \times V \to A$ such that for all $v,w \in V$ and $a,b \in A$, $h(va,wb) = \sigma(a)h(v,w)b$,  $h(v,w) = \sigma(h(w,v))$ and $h$ is bi-additive. We will also assume that all hermitian forms satisfy a non-degeneracy condition, that is to say, for all $v \in V-\{0\}$ there is a $w \in V-\{0\}$ such that $h(v,w) \neq 0$.

We associate to any hermitian form $h$ over an algebra with involution $(A, \sigma)$, the adjoint involution $\tau_h$ on the space of endomorphisms of $V$ over $A$. This association gives a bijective correspondence between hermitian forms on $V$ modulo factors in $k^*$ and involutions on $\End_A(V)$ whose restriction to $K$ is $\sigma$. Since by Wedderburn's theorem \cite[Theorem 2.1.3]{GilleSzam} we may write $A$ as $\End_D V$ for $V$ a vector space over a division algebra $D$, we may write any central simple algebra with involution $(A, \sigma)$ as $(\End_D V, \tau_{h})$ where $h$ is a hermitian form over $(D, \theta)$ and $\theta$ is an involution whose restriction to $K$ is $\sigma$.

If $a$ is an algebraic element over $k$, consider the $k$-linear map $s:k(a) \to k$given by $s(1)=1$ and $s(a^j) = 0$ for all $1 \leq j < m$ where $m = [k(a):k]$. The map $s$ induces a transfer homomorphism $s_*$ from the Witt group of hermitian forms over $(D_{k(a)}, \theta_{k(a)})$ to the Witt group of hermitian forms over $(D, \theta)$. We will refer to this homomorphism as Scharlau's transfer homomorphism. Bayer and Lenstra have shown \cite{BayerLenstra} that if $[k(a):k]$ is odd, $r^*$ is the extension of scalars from $k$ to $k(a)$, and $h$ is a hermitian form over $(D, \theta)$ then $s_*(r^*(h)) = h$ in $W(D, \theta)$. We may regard $W(D_{k(a)}, \theta_{k(a)})$ as a $W(k(a))$-module. For example, we may write $ar^*(h) \in W(D_{k(a)}, \theta_{k(a)})$ as $\qform{a} \otimes r^*(h) \in W(k(a)) \otimes W((D_{k(a)}, \theta_{k(a)})$. Bayer and Lenstra have shown \cite{BayerLenstra} that $s_*(\qform{a} \otimes r^*(h))= s_*(\qform{a}) \otimes h$. 

\section{Properties of Algebraic Groups} \label{SG}

Let $k$ be a field of characteristic different from 2. A simply connected (respectively adjoint), semisimple algebraic $k$-group $G$ is a product of groups of the form $R_{E_j/k}(G_j)$ where each $E_j$ is a finite, separable extension of $k$ and each $G_j$ is an absolutely simple, simply connected (respectively adjoint) group \cite[Theorem 26.8]{KMRT}. We call $G$ a \emph{classical} group if each $G_j$ is classical. 

An absolutely simple, simply connected, classical $k$-group $G$ has one of the following forms  \cite[26.A.]{KMRT}, \cite{BayerParimala}:

\begin{description}
\item [Type $^1A_{n-1}$] $G=SL_1(A)$ for a central simple algebra $A$ of degree $n$ over $k$.
\item [The unitary case] $G=SU(A, \sigma)$ associated to a central simple algebra $A$ over $K$ of degree $n$ at least 2, with $\sigma$ a unitary involution on $A$ with $K^\sigma = k$.
\item [The symplectic case] $G=Sp(A,\sigma)$ associated to a central simple algebra $A$ over $k$ of even degree with a symplectic involution $\sigma$.
\item [The orthogonal case] $G=Spin(A, \sigma)$ associated to a central simple algebra $A$ over $k$ of degree at least 3, with $\sigma$ an orthogonal involution on $A$.
\end{description}

Let $k$ be field of characteristic different from 2. An absolutely simple, adjoint, classical $k$-group $G$ has one of the following forms: \cite[26.A.]{KMRT}:

\begin{description}
\item [Type $^1A_{n-1}$] $G=PGL_1(A)$ for a central simple algebra $A$ of degree $n$ over $k$.
\item [The unitary case]  $G=PGU(A, \sigma)$ associated to $A$ is a central simple algebra over a field $K$ of degree $n$ at least 2 and $\sigma$ is a unitary involution on $A$ with $K^{\sigma}=k$.
\item [The symplectic case]  $G=PGSp(A,\sigma)$ associated to a central simple algebra $A$ over $k$ of even degree and $\sigma$ a symplectic involution on $A$.
\item [The orthogonal case] We distinguish between groups of type $B$ and $D$ in this case.

\begin{description}
\item[Type $B_n$] $G=O^+(A, \sigma)$ associated to a central simple algebra $A$ over $k$ of odd degree at least 3 and $\sigma$ an orthogonal involution on $A$.
\item[Type $D_n$] $G=PGO^+(A, \sigma)$ associated to a central simple algebra $A$ over $k$ of even degree at least 4 and $\sigma$ an orthogonal involution on $A$.
\end{description}
\end{description}

In the classification of semisimple algebraic groups, exceptional groups are precisely those of types $^{3,6}D_4 , E_6 , E_7 , E_8 , F_4$ and $G_2$ and a group of type $E_8$, $F_4$ or $G_2$ is both simply connected and adjoint. 

In \cite{Serre} Serre defines a set of primes $S(G)$ associated to an absolutely simple $k$-group $G$ which we will refer to as the \emph{homological torsion primes} of $G$. $S(G)$ is the set of prime numbers $p$ each of which satisfies one of the following conditions:
\begin{enumerate}
\item $p$ divides the order of the automorphism group of the Dynkin graph of $G$
\item $p$ divides the order of the center of the universal cover of $G$
\item $p$ is a torsion prime of the root system of $G$
\end{enumerate}

The values of the cohomological torsion primes for each of the absolutely simple semisimple groups is as follows:

\begin{center} \begin{tabular}{| c | c | } 
\hline Group &S(G) \\ \hline \hline
\hline Type $^1A_{n-1}$ & prime divisors of $n$ \\ \hline 
unitary case & 2, prime divisors of $n$\\ \hline 
symplectic case &2 \\ \hline
orthogonal case &2 \\ \hline
$G_2$ &2\\ \hline
$F_4$ &2,3\\ \hline
$^{3,6}D_4$, $E_6$, $E_7$ & 2,3 \\\hline 
$E_8$ & 2,3,5\\\hline
\end{tabular} 
\end{center}

\bigskip

We mention that for each absolutely simple group $G$, the prime factors of the Dynkin index of $G$ are contained in $S(G)$ \cite{SkipMerkSerre}. 

\section{Galois Cohomology of Algebraic Groups} \label{GC}

Let $k$ be a field. For an algebraic $k$-group $G$, let $H^i(k, G) = H^i(\Gamma_k, G(k_s))$ denote the Galois Cohomology of $G$ with the assumption $i \leq 1$ if $G$ is not abelian. For any $k$-group $G$, $H^0(k, G)=G(k)$ and $H^1(k,G)$ is a pointed set which classifies the isomorphism classes of principal homogeneous spaces under $G$ over $k$. The point in $H^1(k,G)$ corresponds to the principal homogeneous space with rational point. We will interchangeably denote the point in $H^1(k,G)$ by \emph{point} or 1. 

For $\text{Gal}(k_s/L)$ a subgroup of $\text{Gal}(k_s/k)$ of finite index $n$, we have a restriction homomorphism $H^1(k,G) \to H^1(L,G)$. If $G$ is abelian, we also have a corestriction homomorphism $H^1(L, G) \to H^1(k, G)$. The composite $\cor \circ \res$ is multiplication by $n$. In particular, the restriction map is injective on the (prime-to-$n$)-torsion part of $H^1(k,G)$. 

Each $\Gamma_k$-homomorphism $f:G \to G^{\prime}$ induces a canonical map $H^i(k, G) \to H^i(k, G^{\prime})$ which we shall also denote by $f$. Given an exact sequence of $k$-groups,
\[
\xymatrix{
1 \ar[r] &G_1 \ar[r]^-{f_1} &G_2 \ar[r]^-{f_2} &G_3 \to 1
}
\]

\bigskip

\noindent
there exists a map of pointed sets $\delta_0:G_3(k) \to H^1(k, G_1)$ such that the following sequence of pointed sets is exact

\[
\xymatrix{
G_1(k) \ar[r]^-{f_1} &G_2(k) \ar[r]^-{f_2} &G_3(k) \ar[r]^-{\delta_0} &H^1(k,G_1) \ar[r]^-{f_1} &H^1(k,G_2) \ar[r]^-{f_2} &H^1(k,G_3)}
\]

\bigskip

\noindent
If in addition $G_1$ is central in $G_2$, there is a connecting map $\delta_1:H^1(k, G_3) \to H^2(k, G_1)$ such that 

\[
\xymatrix{
G_3(k) \ar[r]^-{\delta_0} &H^1(k,G_1) \ar[r]^-{f_1} &H^1(k,G_2) \ar[r]^-{f_2} &H^1(k,G_3) \ar[r]^-{\delta_1} &H^2(k,G_1)
}
\]

\noindent
is an exact sequence of pointed sets.

To conclude this section we list some well-known results on the Galois Cohomology of certain groups $G$ which we shall use often. 

\bigskip

\begin{description}
\item [Hilbert's Theorem 90] For a separable, associative $k$-algebra $R$, let $GL_1(R)$ be the algebraic group whose group of $A$-points, for any commutative $k$-algebra $A$ is $(A \otimes_k R)^*$. For example, $GL_1(k)$ is just the multiplicative group $G_m$. Then $H^1(k, GL_1(R)) = 1$. The \emph{classical Hilbert 90} states that if $L/k$ is a cyclic Galois extension of fields with Galois group generated by $\theta$ then any element $\alpha$ of $L^*$ with $N_{L/k}(\alpha) =1$ is of the form $\mu^{-1} \theta(\mu)$ for some $\mu \in L^*$. This classical result follows from the more general statement preceding it.
\item [Cohomological Brauer group] $H^2(k, G_m)$ is the Brauer group of $k$ and for $\mu_n$ the group of $n$-th roots of unity, $H^2(k, \mu_n)$ is the $n$-torsion subgroup of the Brauer group of $k$.
\item [Kummer theory] $H^1(k, \mu_n) = k^*/(k^*)^n$.
\item [Shapiro's Lemma]Let $G$ be an algebraic $k$-group and let $L/k$ be a finite extension of groups. Let $R_{L/k}(G)$ denote the Weil restriction of $G$ given by $R_{L/k}(G)(A)=G(A \otimes_k L)$ for any commutative $k$-algebra $A$. Then $H^i(k,R_{L/k}(G)) = H^i(L,G)$.
\end{description}

 \section{Absolutely Simple Simply Connected Groups of Classical Type} \label{SC}

The main result of this section is the following:

\begin{theorem}\label{scmeta} Let $k$ be a field of characteristic different from 2. Let $G$ be an absolutely simple, simply connected, classical algebraic group over $k$. Let $\{L_i\}_{1 \leq i \leq m}$ be a set of finite field extensions of $k$ and let the greatest common divisor of the degrees of the extensions$[L_i:k]$ be $d$. If $d$ is coprime to $S(G)$, then the canonical map
\[
H^1(k,G) \to \prod_{i=1}^m H^1(L_i,G)
\]
has trivial kernel.
\end{theorem}

If $G=Sp(A,\sigma)$ for $\sigma$ a symplectic involution on a central simple algebra over $k$, then $H^1(k, Sp(A, \sigma))$ classifies rank one hermitian forms over $(A, \sigma)$. Then for any finite extension of odd degree $L$ over $k$, triviality of the kernel of the map $H^1(k,Sp(A, \sigma)) \to H^1(L,Sp(A, \sigma))$ is a consequence of the Bayer-Lenstra theorem \cite[Theorem 2.1]{BayerLenstra}. We discuss the remaining cases in what follows.

We will need the following lemma in the rest of this section.

\begin{lemma} \label{index}
Let $K$ be a field and let $A$ be a central simple algebra over $K$ of index $s$. Let $\Nrd$ be the reduced norm. For every $\alpha \in K^*$, there exists $\beta \in A^*$ such that $\Nrd(\beta) = \alpha^s$
\end{lemma}

\begin{proof}
By \cite[Proposition 4.5.4]{GilleSzam}, choose a splitting field $E$ for $A$ such that $[E:K] = s$. Since $A_E$ is split, $\Nrd:A_E \to E$ is onto. In particular $\alpha$ is in $\Nrd(A_E)$. Since $N_{E/K}(\Nrd(A_E)) \subset \Nrd(A)$ \cite[Corollary 2.3]{BayerPariHasse} and $N_{E/K}(\alpha) = \alpha^s$, it follows that $\alpha^s$ is in $\Nrd(A)$.
\end{proof}

\subsection*{Type $^1A_{n-1}$}

\begin{proposition}
Let $k$ be a field, A a central simple algebra of degree $n$ over $k$ and $G=SL_1(A)$. Let $\{L_i\}_{1 \leq i \leq m}$ be a set of finite extensions of $k$ let $\gcd([L_i:k])=d$. If $d$ is coprime to $n$, then the canonical map  
\[
H^1(k,G) \to \displaystyle \prod_{i=1}^m H^1(L_i, G)
\]
has trivial kernel

\end{proposition}

\begin{proof}
Consider the short exact sequence
\begin{equation} 
\xymatrix{
1 \ar[r] &SL_1(A) \ar[r] &GL_1(A) \ar[r]^-{\Nrd} &G_m \ar[r] &1}
\end{equation}

\noindent
which by Hilbert's Theorem 90 induces the following commutative diagram with exact rows.

\begin{equation} 
\xymatrix{
A^* \ar[r]^-{Nrd} \ar[d] &k^* \ar[r]^-\delta \ar[d]^-g &H^1(k,SL_1(A)) \ar[r] \ar[d]^-h &1 \\
\prod A_{L_i}^* \ar[r]^-{Nrd} & \prod L_i^* \ar[r]^-\delta  &\prod H^1(L_i,SL_1(A)) \ar[r] &1}
\end{equation}

Choose $\lambda \in \ker(h)$. By the exactness of the top row of the diagram, choose $\lambda^\prime \in k^*$ such that $\delta(\lambda^\prime) = \lambda$. Fix an index $i$. Since $\delta(g(\lambda^\prime)) = \point$, by exactness of the bottom row choose $(\lambda^{\prime\prime}_i) \in A_{L_i}^*$ such that $\Nrd(\lambda^{\prime\prime}_i) = g(\lambda^\prime)$. By restriction-corestriction, $N_{L_i/k}(g(\lambda^\prime)) = ({\lambda^\prime})^{m_i}$ where $m_i=[L_i:k]$. By the norm principle for reduced norms \cite[Corollary 2.3]{BayerPariHasse}, $N_{L_i/k}(\Nrd(A_{L_i}^*)) \subset \Nrd(A^*)$. In particular, $({\lambda^\prime})^{m_i}$ is in $\Nrd(A^*)$. Since $d = \sum m_i n_i$ for appropriate choice of integers $n_i$, $(\lambda^\prime)^d = \prod ((\lambda^\prime)^{m_i})^{n_i}$ is in $\Nrd(A^*)$.

Let $s$ be the index of $A$. Then by \ref{index}, $(\lambda^\prime)^s$ is in $\Nrd(A^*)$. Since $s$ divides $n$ and by assumption $d$ and $n$ are coprime, then $d$ and $s$ are coprime. So choose $a$ and $b$ such that $sa+db=1$. Then $\lambda^\prime = ({\lambda^\prime})^{sa} ({\lambda^\prime})^{db}$ is in $\Nrd(A^*)$ and by exactness of the top row $\lambda=\delta(\lambda^\prime)$ is the point in $H^1(k,SL_1(A))$.
\end{proof}

\subsection*{The Unitary case}
\begin{theorem}\label{SU}
Let  A be a central simple algebra of degree $n$ with center K and $\sigma$ a unitary involution on $A$ with $K^{\sigma}=k$. Suppose $deg_K (A )\geq 2$. Let  $G=SU(A, \sigma)$. Let $\{L\}_{1 \leq i \leq m}$ be a set of finite field extensions of $k$ with $\gcd([L_i:k]) =d$. If $d$ is odd and coprime to $n$, then the canonical map
\[
H^1(k,G) \rightarrow \displaystyle \prod_{i=1}^m H^1(L_i,G)
\]
has trivial kernel.

\end{theorem}

\begin{proof}
Consider the short exact sequence
\begin{equation} 
\noindent\
\xymatrix{
 1 \ar[r] &SU(A, \sigma) \ar[r] &U(A, \sigma) \ar[r]^-{Nrd} &R^1_{K/k}G_m \ar[r] &1
}
\end{equation}

\noindent
which induces the following commutative diagram in Galois Cohomology with exact rows.

\begin{equation} \label{SUDynkin}
\xymatrix{
U(A, \sigma)(k) \ar[r]^-{Nrd} \ar[d] &R^1_{K/k}G_m(k) \ar[r]^-{\delta} \ar[d]^-f & H^1(k, SU(A, \sigma)) \ar[r]^j \ar[d]^-g &H^1(k, U(A, \sigma)) \ar[d]^-h \\
\prod U(A, \sigma)(L_i) \ar[r]^-{Nrd} &\prod R^1_{K/k}G_m(L_i) \ar[r]^-{\delta} & \prod H^1(L_i,SU(A,\sigma)) \ar[r]^j & \prod H^1(L_i, U(A, \sigma))}
\end{equation}

Choose  $\lambda \in \ker(g)$. By assumption, there is an index $i$ such that $[L_i:k]$ is odd. Fix that index $i$ and let $L_i=L$. By the Bayer-Lenstra theorem \cite[Theorem 2.1]{BayerLenstra}, $H^1(k,U(A, \sigma)) \to H^1(L,U(A, \sigma))$ has trivial kernel. In particular, $h$ has trivial kernel and $\lambda$ is in $\ker(j)$. So choose $\lambda^{\prime}  \in R^1_{K/k} G_{\mathbf{m}} (k)$ such that $\delta(\lambda^\prime) = \lambda$. Since $\delta(f(\lambda^\prime)) = \text{point}$, exactness of the bottom row of the diagram gives $(\lambda^{\prime\prime}_i) \in \prod U(A, \sigma)(L_i) $ such that $\Nrd(\lambda^{\prime\prime}_i) = f( \lambda^\prime)$. Applying $N_{L_i/k}$ to both sides of this equality we find $N_{L_i/k}(\Nrd(\lambda^{\prime\prime}_i)) = N_{L_i/k}(f( \lambda^\prime))$. Since $U(A, \sigma)$ is a rational group, \cite[Theorem 3.9]{MerkurjevNP} gives that for each $i$, $N_{L_i/k}(\Nrd(\lambda^{\prime\prime}_i))$ is in the image of $\Nrd: U(A, \sigma)(k) \to R^1_{K/k} G_m(k)$. By restriction-corestriction, for each $i$, $N_{L_i/k}(f( \lambda^\prime))=(\lambda^\prime)^{m_i}$ for $m_i=[L_i:k]$. So for each $i$, $(\lambda^\prime)^{m_i}$ is in the image of $\Nrd: U(A, \sigma)(k) \to R^1_{K/k} G_m(k)$. Since $(\lambda^\prime)^d = \prod((\lambda^\prime)^{m_i})^{n_i}$ for appropriate choice of integers $n_i$, then $(\lambda^\prime)^d$ is in the image of $\Nrd: U(A, \sigma)(k) \to R^1_{K/k} G_m(k)$.

By Classical Hilbert 90 write $\lambda^\prime= \mu^{-1} \bar \mu$ for $\mu \in K^*$ and $\bar \mu$ the image of $\mu$ under the nontrivial automorphism of $K$ over $k$. Let $s$ be the index of $A$ and write $(\lambda^\prime)^s= (\mu^s)^{-1} \bar{\mu^s}$. By \ref{index}, $\mu^s=\Nrd(a)$ for some $a \in A*$. Thus $(\lambda^\prime)^s = \Nrd(a^{-1} \sigma(a))$ and by Merkurjev's theorem \cite[Proposition 6.1]{MerkurjevNP} $(\lambda^\prime)^s$ is in the image of $\Nrd:U(A, \sigma)(k) \to R^1_{K/k} G_m(k)$.

Certainly, $s$ divides $n$ and since by assumption $d$ is coprime to $n$, then $d$ is coprime to $s$. In particular, there exist $v,w \in \mathbb{Z}$ such that $dv+sw=1$. Therefore $\lambda^\prime = (\lambda^\prime)^{dv} (\lambda^\prime)^{sw}$ is in the image of $\Nrd: U(A, \sigma)(k) \to R^1_{K/k} G_m(k)$ and by exactness of the top row of \eqref{SUDynkin}, $\lambda = \delta(\lambda^\prime) = \point$.

\end{proof}

\subsection*{The Orthogonal case} ~\\

\noindent 
Our proof in this case makes use of the following result. 

\begin{proposition} \label{sub}
Let $k$ be a field of characteristic different from 2 and let $A$ be a central simple algebra over $k$ of degree $\geq 3$ with orthogonal involution $\sigma$. Let $G=O^+(A, \sigma)$ and let $L$ be a finite extension of $k$ of odd degree. Then the canonical map  
\[ 
H^1(k, G) \rightarrow H^1(L,G)
\]
has trivial kernel.

\end{proposition}
\begin{proof}

We have the short exact sequence 
\begin{equation} \label{SUh1SES}
\noindent\
\xymatrix{
 1 \ar[r] &O^+(A, \sigma) \ar[r] &O(A, \sigma) \ar[r]^-{Nrd} &\mu_2 \ar[r] &1
}
\end{equation}

In the case $A$ is split, $O(A, \sigma) = O(q)$ the orthogonal group of a quadratic form $q$, $O^+(A, \sigma) = O^+(q)$ and the reduced norm is the determinant. Springer's theorem \cite[Chapter VII Theorem 2.7]{Lam} gives $H^1(k, O(q)) \to H^1(L, O(q))$ has trivial kernel. That $H^1(k, O^+(q)) \to H^1(k, O(q))$ has trivial kernel follows from the observation that the determinant map $O(q)(k) \to \mu_2$ is onto. Combining these two results, \ref{sub} holds. 

So assume $A$ is not split. Then $O^+(A, \sigma)(k) = O(A, \sigma)(k)$ \cite[2.6, Lemma 1.b]{Kneser}. Since $A$ admits an involution of the first kind and $L/k$ is odd, $A_L$ is not split and $O^+(A, \sigma)(L)=O(A, \sigma)(L)$.

Then \eqref{SUh1SES} induces the following diagram with exact rows and commuting rectangles.

\begin{equation} \label{SUh1LES}
\xymatrix{
 1 \ar[r] &\mu_2 \ar[r]^-\delta \ar[d]^h &H^1(k, O^+(A, \sigma)) \ar[r]^-i \ar[d]_-f &H^1(k, O(A, \sigma)) \ar[d]^-g \\
 1 \ar[r] &\mu_2 \ar[r]^-\delta &H^1(L, O^+(A, \sigma)) \ar[r]^-i &H^1(L, O(A, \sigma))}
\end{equation}

\bigskip

\noindent
Let $\lambda \in \ker(f) $. By the commutativity of the rightmost rectangle in \eqref{SUh1LES}, $g(i(\lambda)) = \point$. Then \cite[Theorem 2.1]{BayerLenstra} gives $i(\lambda) = \point$. By the exactness of the top row, there exists $\lambda^\prime \in \mu_2$ such that $\delta(\lambda^\prime) = \lambda$. Since the left rectangle in \ref{SUh1LES} commutes, $\delta(h(\lambda^\prime)) = \point$. Since $h$ is the identity map and $\delta$ has trivial kernel, $\lambda^\prime=1$ and thus $\lambda=\delta(\lambda^\prime) = \point$.

\end{proof}

Now we give the proof for absolutely simple, simply connected groups in the orthogonal case.

\begin{theorem} \label{Spinh}
Let $k$ be a field of characteristic different from 2 and let $A$ be a central simple algebra over $k$ of degree $\geq 4$ with orthogonal involution $\sigma$. Let $G=Spin(A, \sigma)$ and let $L$ be a finite extension of $k$ of odd degree. Then the canonical map  
\[
H^1(k, G) \rightarrow H^1(L,G) 
\]
has trivial kernel.

\end{theorem}

\begin{proof}
The short exact sequence 

\begin{equation} \label{SpinhSES}
\noindent\
\xymatrix{
 1 \ar[r] &\mu_2 \ar[r]^-i &Spin(A, \sigma) \ar[r]^-{\eta} &O^+(A, \sigma) \ar[r] &1
}
\end{equation}

induces the following commutative diagram with exact rows. 

\begin{equation} \label{SpinhLES}
\xymatrix{
 O^+(A, \sigma)(k) \ar[d]^-f \ar[r]^-{\delta} &H^1(k,\mu_2) \ar[r]^-i \ar[d]^g &H^1(k, Spin(A, \sigma)) \ar[r]^-\eta \ar[d]_-h &H^1(k, O^+(A, \sigma)) \ar[d]^-j \\
 O^+(A, \sigma)(L) \ar[r]^-{\delta} &H^1(L,\mu_2) \ar[r]^-i &H^1(L, Spin(A, \sigma)) \ar[r]^-\eta &H^1(L, O^+(A, \sigma))}
\end{equation}

\bigskip

\noindent
Choose $\lambda \in \ker(h)$. By commutativity of the rightmost rectangle in \eqref{SpinhLES} $j(\eta(\lambda)) = \point$. In particular, $\eta(\lambda) \in \ker(j)$ and by  \ref{sub}, $\eta(\lambda) = \point$. By exactness of the top row, we may choose $\lambda^\prime \in H^1(k, \mu_2)$ such that $i(\lambda^\prime) = \lambda$. By the commutativity of the central rectangle in \eqref{SpinhLES}, $i(g(\lambda^\prime)) = \point$. So from exactness of the bottom row, we may choose $\lambda^{\prime\prime} \in O^+(A, \sigma)(L)$ such that $\delta(\lambda^{\prime\prime}) =g(\lambda^\prime)$. Applying the norm map to both sides of this equality we find, $N_{L/k}(\delta(\lambda^{\prime\prime})) = N_{L/k}(g(\lambda^\prime))$. By restriction-corestriction the latter is $(\lambda^\prime)^{[L:k]}$. Let $\tilde \lambda$ be a representative of $\lambda^{\prime}$ in $k^*/(k^*)^2$. Since $[L:k]$ is odd, $\tilde \lambda^{[L:k]} = \tilde \lambda$ in $k^*/(k^*)^2$ . In turn $[(\lambda^{\prime})^{[L:k]}]= [\lambda^{\prime}]$ in $H^1(k, \mu_2)$. Thus $N_{L/k}(\delta(\lambda^{\prime\prime})) = \lambda^{\prime}$. Since $O^+(A, \sigma)$ is rational, \cite[Th\'eor\`eme II.3.2]{GilleNP} gives \[\xymatrix{N_{L/k}(\text{im} (O^+(A, \sigma)(L) \ar[r]^-{\delta} &H^1(L,\mu_2)) \subset \text{im} (O^+(A, \sigma)(k) \ar[r]^-{\delta} &H^1(k,\mu_2)))}.\]
In particular  $\lambda^{\prime}$ is in the image of $O^+(A, \sigma)(k) \to H^1(k, \mu_2)$. But then by exactness of the top row, $\lambda=i(\lambda^\prime) = \point$.
\end{proof} 

\section{Absolutely Simple Adjoint Groups of Classical Type} \label{AD}

We begin by recording some general results which we shall use in the proof of \ref{admeta}. 

\begin{proposition} \label{Sim}
Let $A$ be a central simple algebra over a field $K$ with involution $\sigma$ of any kind and $k=K^{\sigma}$. Let $L$ be a finite extension of $k$ of odd degree. Let $G$ be the group of similitudes of $(A, \sigma)$. Then the canonical map
\[
H^1(k,G) \to H^1(L,G)
\]
has trivial kernel.
\end{proposition}

\begin{proof}
Let $G_0$ be the group of isometries of $(A, \sigma)$. We have the exact sequence
\[
1 \to G_0 \to G \to G_m \to 1
 \]

\noindent
where the map $G \to G_m$ takes each similitude $a$ to its multiplier $\sigma(a) a$. In view of Hilbert's Theorem 90, the sequence yields the following commutative diagram with exact rows.
\begin{equation} \label{SimLES}
\xymatrix{
k^* \ar[r]^-{\delta} \ar[d] &H^1(k, G_0) \ar[r]^-i \ar[d]^-{r^*} &H^1(k, G) \ar[d]^-g \ar[r] &1\\
L^* \ar[r]^-\delta &H^1(L,G_0) \ar[r]^i &H^1(L, G) \ar[r] &1
}
\end{equation}

Let $\psi \in \ker(g)$. By the exactness of the top row of \eqref{SimLES}, there exists $\qform{x} \in H^1(k, G_0)$ such that $i(\qform{x})= \psi$. Here $\qform{x}$ is a rank one hermitian form over $(A, \sigma)$. Since commutativity of the right rectangle gives $i(r^*(\qform{x})) =$ point, exactness of the second row gives an $a \in L^*$ such that $r^*(\qform{x}) = \delta(a)$. We note that $\delta(a)$ is the isomorphism class of the rank one hermitian form $\qform{a}$ over $(A, \sigma)_L$. 

Let $k(a)$ be the subfield of $L$ generated by $a$ over $k$. Since $L$ is an odd degree extension of $k(a)$ and $\qform{a}_L \cong r^*(\qform{x})_L$ then $\qform{a}_{k(a)} \cong r^*(\qform{x})_{k(a)}$ \cite[Corollary 1.4]{BayerLenstra}. Let $s:k(a) \to k$ be the $k$-linear map given by $s(1)=1$ and $s(a^j) = 0$ for all $1 \leq j < m$ where $m = [k(a):k]$ and let $s_*$ be the induced transfer homomorphism. Write $\qform{a}$ as $\qform{a} \otimes \qform{1}_{k(a)}$ in $W(k(a)) \otimes W(A_{k(a)},\sigma_{k(a)})$. Since $[k(a):k]$ is odd, results of Bayer-Lenstra and Scharlau give that $s_*(\qform{a} \otimes \qform{1}_{k(a)})$ is Witt equivalent to $\qform{N_{k(a)/k}(a)} \otimes \qform{1}$ \cite[Chapter 2, Lemma 5.8]{Scharlau}\cite[p 362]{BayerLenstra}. On the other hand, $s_*(r^*(\qform{x})) = s_*(\qform{1} \otimes \qform{x})$ and since $[L:k]$ is odd, $s_*(\qform{1} \otimes \qform{x}) \cong \qform{x}$ \cite{Scharlau}. So $\qform{N_{k(a)/k}(a)}$ is Witt equivalent to $\qform{x}$ and since the two forms have dimension one over $(A, \sigma)$, by Witt's cancellation for hermitian forms, $\qform{N_{k(a)/k}(a)} \cong \qform{x}$. Then $\qform{x} = \delta(N_{k(a)/k}(a))$ and thus $\psi = i(\qform{x}) = \point$.
\end{proof}

The following is a straightforward corollary of \ref{Sim}.

\begin{proposition} \label{PSim}
Let $A$ be a central simple algebra over a field $k$ with an involution $\sigma$ of the first kind. Let $G$ be the group of projective similitudes of $(A, \sigma)$. Let $L$ be a finite extension of $k$ of odd degree. Then the canonical map
\[
H^1(k, G) \to H^1(L, G)
\]
has trivial kernel
\end{proposition}

\begin{proof}

Let $G_0$ be the group of similitudes $Sim(A, \sigma)$. Then we have the short exact sequence 
\begin{equation}
\xymatrix{
1 \ar[r] &G_m \ar[r]^-i &G_0 \ar[r]^-{\eta} &G \ar[r] &1
}
\end{equation}

\noindent
which induces the following commutative diagram with exact rows. 

\begin{equation} \label{PSimLES}
\xymatrix{
1 \ar[r] &H^1(k,G_0) \ar[r]^-\eta \ar[d]^-f &H^1(k,G) \ar[r]^{\delta} \ar[d]^-{g} &H^2(k, G_m) \ar[d]^-h\\
1 \ar[r] &H^1(L,G_0) \ar[r]^-\eta & H^1(L,G) \ar[r]^-\delta &H^2(L,G_m)
}
\end{equation}

The set $H^1(k,G)$ is in bijection with the isomorphism classes of central simple algebras of the same degree as $A$ with involution of the same type as $\sigma$. Choose $(A^\prime, \sigma^\prime) \in \ker(g)$. By commutativity of \eqref{PSimLES}, $h(\delta(A^\prime, \sigma^\prime)) = \point$. Now $\delta(A^\prime, \sigma^\prime) = [A^\prime][A]^{-1}$ \cite[pg 405]{KMRT} which is 2-torsion in the Brauer group by \cite[Theorem 3.1]{KMRT}. In particular, since $[L:k]$ is odd, $h$ has trivial kernel on the image of $\delta$ and $\delta(A^\prime, \sigma^\prime) = \point$. By exactness of the top row of the diagram, choose $\lambda^\prime \in H^1(k,G_0)$ such that $\eta(\lambda^\prime)=(A^\prime, \sigma^\prime)$. By choice of $\lambda^\prime$, $\eta(f(\lambda^\prime)) = \point$. Then by exactness of the bottom row $f(\lambda^\prime) = \point$. But that $f$ has trivial kernel was shown in \ref{Sim}. So $\lambda^\prime =  \point$ and in turn $(A^\prime, \sigma^\prime)$ is the point in $H^1(k,G)$.
 
\end{proof}

Next we prove a norm principle for multipliers of similitudes.

\begin{lemma} \label{norm}
Let $A$ be a central simple $K$-algebra with $k$-linear involution $\sigma$. Let $L$ be a finite extension of $k$ of odd degree and let $g$ be a similitude of $(A, \sigma)_L$ with multiplier $\mu(g)$. Then $N_{L/k}(\mu(g))$ is the multiplier of a similitude of $(A, \sigma)$
\end{lemma}
\begin{proof}

Let $g$ be a similitude of $(A, \sigma)_L$. Let $\mu(g)=\sigma(g)g$ be the multiplier of $g$. By definition, the hermitian form $\qform{\mu(g)}_L$ is isomorphic to $\qform{1}_L$. In particular left multiplication by $g$ gives an explicit isomorphism between the hermitian forms. We may identify $\qform{\mu(g)}_L$ with $\qform{\mu(g)}_L \otimes \qform{1}_L$ in $W(L) \otimes W(A_L, \sigma_L)$. Since $[L:k(\mu(g))]$ is odd and $\qform{\mu(g)}_L \otimes \qform{1}_L \cong \qform{1}_L$  then $\qform{\mu(g)}_{k(\mu(g))} \otimes \qform{1}_{k(\mu(g))} \cong \qform{1}_{k(\mu(g))}$ \cite[Corollary 1.4]{BayerLenstra}. Let $s$ be Scharlau's transfer map from $k(\mu(g)) \to k$ and let $s_*$ be the induced transfer homorphism. Then $s_*(\qform{\mu(g)}_{k(a)} \otimes \qform{1}_{k(a)})$ is Witt equivalent to $\qform{N_{k(\mu(g))/k}(\mu(g))}\otimes \qform{1} $ \cite[Chapter 2, Lemma 5.8]{Scharlau}\cite[p 362]{BayerLenstra}. Since on the other hand $s_*(\qform{1}_{k(\mu(g))}) = \qform{1}$, then  $\qform{N_{k(\mu(g))/k}(\mu(g))} \otimes \qform{1} $ is Witt equivalent to $1$. Since both are rank 1 hermitian forms, it follows from Witt's cancellation that they are in fact isomorphic which gives precisely that $N_{k(\mu(g))/k}(\mu(g))$ is the multiplier of a similitude of $(A, \sigma)$.
\end{proof}

Having established these results we move on to the main result of this section.

\begin{theorem}\label{admeta}
Let $k$ be a field of characteristic different from 2 and $G$ an absolutely simple, adjoint, classical group over $k$.  Let $\{ L_i\}_{1 \leq i \leq m}$ be a set of finite field extensions of k and let the greatest common divisor of the degrees of the extensions $[L_i:k]$ be d. If $d$ is coprime to $S(G)$ the canonical map
\[
H^1(k,G) \rightarrow \displaystyle\prod_{i=1}^m H^1(L_i,G)
\]
has trivial kernel.

\end{theorem}

The symplectic case $G=PGSp(A, \sigma)$ is just a special case of \ref{PSim} above. So to prove \ref{admeta} it is enough to consider the group of type $^1A_{n-1}$, the orthogonal case and the unitary case. 

\subsection*{Type $^1A_{n-1}$}

\begin{theorem}\label{PGL1}
Let $k$ be a field of characteristic different from 2, $A$ a central simple algebra of degree $n$ over $k$ and $G=PGL_1(A)$. Let $\{ L_i\}_{1 \leq i \leq m}$ be a set of finite field extensions of k and let  $\gcd([L_i:k])=d$. If $d$ is coprime to $n$ then the canonical map
\[
H^1(k,G) \rightarrow \displaystyle\prod_{i=1}^m H^1(L_i,G)
\]
has trivial kernel.
\end{theorem}

\begin{proof}
Consider the short exact sequence
\begin{equation} \label{PGL1SES}
\noindent\
\xymatrix{
 1 \ar[r] &G_{\mathbf{m}} \ar[r] &GL_1(A) \ar[r] &PGL_1(A) \ar[r] &1
}
\end{equation}

\noindent
Since Hilbert's Theorem 90 gives $H^1(k,GL_1(A)) =1$, the induced long exact sequences in Galois Cohomology produces the following commutative diagram with exact rows.

\begin{equation} \label{PGL1LES}
\xymatrix{
 1 \ar[r] &H^1(k,PGL_1(A)) \ar[r]^-\delta \ar[d]^-f &H^2(k, G_{\mathbf{m}}) \ar[d]_-g\\
 1 \ar[r] &\prod H^1(L_i,PGL_1(A)) \ar[r]^-\delta  &\prod H^2(L_i, G_{\mathbf{m}})}
\end{equation}

\bigskip

\noindent
The pointed set $H^1(k,PGL_1(A))$ classifies isomorphism classes of central simple algebras of degree $n$ over $k$ and for $B \in H^1(k,PGL_1(A))$, $\delta(B) = [B][A]^{-1}$. Choose $B \in \ker(f)$. By commutativity of the diagram, $g(\delta(B)) = \point$ in $\prod H^2(L_i,G_m)$.

Let $A^o$ denote the opposite algebra of $A$ and choose $B \otimes A^{o}$ a representative for the class $[B][A]^{-1}$ in $H^2(k,G_m)$. Let the exponent of $B \otimes A^o$ be $s$. Since by assumption $B \otimes A^o$ splits over each $L_i$, $s$ divides each $[L_i:k]$. It follows that $s$ divides $d$. Since the degree of $B \otimes A^{o}$ is $n^2$, $s$ divides $n^2$.

Since by assumption $n$ and $d$ are coprime, $s=1$, $B \otimes A^o$ is split and $B$ is Brauer equivalent to $A$. Then since $B$ and $A$ are of the same degree, they are isomorphic and $B$ is the point in $H^1(k,PGL_1(A))$.
\end{proof}

\subsection*{The unitary case}

\begin{theorem} \label{PGU}
Let $A$ be a central simple algebra of degree $n$ over $K$ with $n \geq 2$, Let $\sigma$ be a unitary involution on $A$.  Let $k$ be the subfield of elements of $K$ fixed by $\sigma$ and $G = PGU(A,\sigma)$. Let $\{L_i\}_{1 \leq i \leq m}$ be a set of finite field extensions of $k$ and let $\gcd([L_i:k])=d$. If $d$ is odd and coprime to $n$ then the canonical map
\[
H^1(k,G) \rightarrow \prod_{i=1}^m H^1(L_i,G)
\]
has trivial kernel.
\end{theorem}

\begin{proof}
We have the short exact sequence 
\begin{equation} \label{PGUSES}
\xymatrix{
 1 \ar[r] &R_{K/k}G_m \ar[r] &GU(A, \sigma) \ar[r] &PGU(A, \sigma) \ar[r] &1
}
\end{equation}

\noindent
Shapiro's Lemma gives $H^1(k, R_{K/k}G_{\mathbf{m}}) \cong H^1(K, G_{\mathbf{m}})$ and the latter is trivial. Therefore, \eqref{PGUSES} induces the following commutative diagram with exact rows.

\begin{equation} \label{PGULES}
\xymatrix{
1 \ar[r] &H^1(k, GU(A, \sigma)) \ar[r]^-{\pi} \ar[d]^-f &H^1(k, PGU(A, \sigma)) \ar[r]^-{\delta} \ar[d]^-g &H^2(k, R_{K/k}G_m) \ar[d]^-h\\
1 \ar[r] &\prod H^1(L_i, GU(A, \sigma)) \ar[r]^-\pi & \prod H^1(L_i, PGU(A, \sigma)) \ar[r] &\prod H^2(L_i, R_{K/k}G_m) 
 }
\end{equation}

Now $H^1(k, PGU(A, \sigma))$ is the set of isomorphism classes of triples $(A^{\prime},\sigma^{\prime}, \phi^{\prime} )$ where $A^{\prime}$ is a central simple algebra over a field $K^{\prime}$ which is a quadratic extension of $k$, the degree of $A^{\prime}$ over $K^\prime$ is $n$, $\sigma^{\prime}$ is a unitary involution on $A^{\prime}$ with $(K^\prime)^{\sigma^{\prime}}=k$, and $\phi^{\prime}$ is an isomorphism from $K^{\prime}$ to $K$.\cite[pg. 400]{KMRT}. 

Choose $(A^{\prime},\sigma^{\prime}, \phi^{\prime}) \in \ker(g)$. Now $\delta(A^{\prime},\sigma^{\prime}, \phi^{\prime}) = [A^\prime \otimes_{K^\prime} K][A]^{-1}$. Let $A^o$ denote the opposite algebra of $A$ and choose $(A^\prime \otimes_{K^\prime} K) \otimes_K A^o$ a representative of $[A^\prime \otimes_{K^\prime} K][A]^{-1}$ in $H^2(k,R_{K/k}G_m)$. Let the exponent of $A\prime \otimes_{K^\prime} K) \otimes_K A^o$ be $t$.

By commutativity of the rightmost rectangle of \eqref{PGULES}, $h(\delta(A^{\prime},\sigma^{\prime}, \phi^{\prime})) = \point$. In particular, a restriction-corestriction argument gives that $m_i \cdot ((A\prime \otimes_{K^\prime} K) \otimes_K A^o)$ is split for each $m_i = [L_i:k]$. Then $t$ divides each $m_i$ and in turn  $t$ divides $d$. 

On the other hand, $t$ divides the degree of $(A^\prime \otimes_{K^\prime} K) \otimes_K A^o$ which is $n^2$. Since $d$ and $n^2$ are by assumption coprime, we find that the exponent of $t$ is 1. Thus $(A\prime \otimes_{K^\prime} K) \otimes_K A^o$ is the point in $H^2(k, R_{K/k}G_m)$. Exactness of the top row of \eqref{PGULES} gives a $\lambda \in H^1(k, GU(A, \sigma))$ such that $\pi(\lambda) = (A^{\prime},\sigma^{\prime}, \phi^{\prime})$. Commutativity of the left rectangle in \eqref{PGULES} gives $\pi(f(\lambda)) = \text{point}$. And thus by the exactness of the bottom row, $f(A^{\prime},\sigma^{\prime}, \phi^{\prime}) = \text{point}$. It follows from \ref{Sim} that $\lambda = \point$ and thus $(A^{\prime},\sigma^{\prime}, \phi^{\prime}) = \point.$
\end{proof}

\medskip

\subsection*{The orthogonal case}~\\

\noindent
The case in which $G$ is of type $B_n$ is a special case of \ref{sub}. For the case in which $G$ is of type $D_n$ we will need the following result on the existence of improper similitudes. 

\begin{lemma} \label{improp}
Let $k$ be a field of characteristic different from 2 and $A$ a central simple algebra over $k$ of even degree at least 4 with an orthogonal involution $\sigma$. Let $L$ be a finite field extension of $k$ of odd degree. If $A$ is not split, then $GO^-(A, \sigma)(k)$ nonempty if and only if  $GO^-(A, \sigma)(L)$ nonempty.
\end{lemma}
 
\begin{proof} 

If $g \in A$ is an improper similitude of $A$ over $k$, then certainly $g_L$ is an improper similitude of $A_L$ over $L$. Conversely, choose $g \in A_L$ an improper similitude of $A_L$ over $L$ and let $\sigma(g)g = \mu(g)$. Then $A_L$ Brauer equivalent to the quaternion algebra $(\delta, \mu(g))$ over $L$ where $\delta$ is the discriminant of $\sigma$ \cite[Theorem 13.38]{KMRT} . From this we find $\cor(A_L)$ Brauer equivalent to $\cor((\delta, \mu(g)))$. Now $\res: H^2(k, \mu_2) \to H^2(L, \mu_2)$ certainly takes $A$ to $A_L$ and $\cor(\res(A)) = A$ since $A$ is 2-torsion and $[L:k]$ is odd. On the other hand, $\cor((\delta, \mu(g)))=(\delta, N_{L/k}(\mu(g)))$. By \ref{norm} write $N_{L/k}(\mu(g))$ as $\mu(g^{\prime})$ for $g^{\prime}$ a similitude of $A$ over $k$. Thus $A$ is Brauer equivalent to $(\delta, \mu(g^{\prime}))$. If $g^{\prime}$ is a proper similitude then by \cite[Proposition 13.38]{KMRT} $(\delta, \mu(g^{\prime}))$ splits. But then $A$ splits and we arrive at a contradiction. So $g^{\prime}$ is an improper similitude of $A$ over $k$. 

 \end{proof} 

\begin{proposition}\label{GO+}
Let $k$ be a field of characteristic different from 2 and $A$ a central simple algebra over $k$ of even degree at least 4 with an orthogonal involution $\sigma$. Let  $G=GO^+(A, \sigma)$ and let $L$ be a finite field extension of $k$ of odd degree. Then the canonical map
\[
H^1(k, G) \rightarrow H^1(L, G)
\]
 has trivial kernel. 
\end{proposition}

\begin{proof}
Consider the short exact sequence 
\begin{equation} \label{GO+SES}
\xymatrix{
1 \ar[r] &GO^+(A, \sigma) \ar[r]^-i &GO(A, \sigma) \ar[r]^-{\eta} &\mu_2 \ar[r] &1}
\end{equation}

\noindent
where the map $\eta$ takes $a \in GO(A, \sigma)$ to $1$ if $\Nrd(a) = \mu(a)^{\text{deg}(A)/2}$ and $\eta^{-1}(-1)$ is precisely $GO^-(A, \sigma)$.
 
 In the case $A$ is split, each hyperplane reflection gives an improper similitude. Thus $GO(A, \sigma)(k) \rightarrow \mu_2$ is onto and \eqref{GO+SES} induces the following commutative diagram with exact rows.

\begin{equation}\label{GO+LES1}
\xymatrix{
1 \ar[r] &H^1(k, GO^+(A,\sigma)) \ar[r]^i \ar[d]^-g &H^1(k, GO(A, \sigma)) \ar[d]^-h\\
1 \ar[r]  &H^1(L, GO^+(A, \sigma)) \ar[r]^i &H^1(L, GO(A,\sigma))}
\end{equation}.

Choose $\lambda \in \ker(g)$.  Since the diagram \eqref{GO+LES1} commutes and $h$ has trivial kernel by \ref{Sim}, $i(\lambda) =$point. Then exactness of the top row of \eqref{GO+LES1} gives $\lambda=$ point.

In the case $A$ is not split, we need only consider two scenarios. Firstly, suppose $A$ and $A_L$ both admit improper similitudes. Then $GO(A, \sigma)(k) \rightarrow \mu_2$ and $GO(A, \sigma)(L) \rightarrow \mu_2$ are both onto and the proof proceeds exactly as in the split case. Otherwise, by \ref{improp} neither admits an improper similitude. That is $GO^+(A, \sigma)(k) =GO(A, \sigma)(k)$, $GO^+(A, \sigma)(L) =GO(A, \sigma)(L)$ and \eqref{GO+SES} induces the following commutative diagram with exact rows. 

\begin{equation}\label{GO+LES2}
\xymatrix{
1 \ar[r] &\mu_2 \ar[r]^-{\delta} \ar[d]^f &H^1(k, GO^+(A,\sigma)) \ar[r]^i \ar[d]^-g &H^1(k, GO(A, \sigma)) \ar[d]^-h\\
1 \ar[r]  &\mu_2 \ar[r]^-{\delta} &H^1(L, GO^+(A, \sigma)) \ar[r]^i &H^1(L, GO(A,\sigma))}
\end{equation}.

Choose $\lambda \in \ker(g)$. Commutativity of the rightmost rectangle in \ref{GO+LES2} gives $i(\lambda) \in \ker(h)$. But by \ref{Sim}, this gives $i(\lambda)= \point$. Then, by exactness of the top row of \eqref{GO+LES2}, $\exists \lambda^{\prime} \in \mu_2$ such that $\delta(\lambda^{\prime}) = \lambda$. Commutativity of the left rectangle in \eqref{GO+LES2} gives $\delta(f(\lambda^{\prime})) =$ point. From whence, since the bottom row of \eqref{GO+LES2} is exact we find $f(\lambda^{\prime}) = 1$. But certainly $f$ is the identity map. So in fact $\lambda^{\prime} = 1$ and in turn, $\lambda = \delta(\lambda^{\prime})=$ point. 
\end{proof}

We may now prove \ref{admeta} for the absolutely simple group in the orthogonal case.

\begin{theorem}\label{PGO+}
Let $k$ be a field of characteristic different from 2 and $A$ a central simple algebra over $k$ of degree at least 4 with an orthogonal involution $\sigma$. Let $G=PGO^+(A, \sigma)$ and let $L$ be a finite field extensions of k of odd degree. Then the canonical map
\[
H^1(k, G) \to H^1(L, G)
\]
 has trivial kernel. 
\end{theorem}

\begin{proof}
Consider the short exact sequence

\begin{equation}\label{PGO+SES}
\xymatrix{
 1 \ar[r] &G_m \ar[r] &GO^+(A, \sigma) \ar[r]^-{\eta} &PGO^+(A, \sigma)  \ar[r] &1
}
\end{equation}

\noindent
by Hilbert's Theorem 90, this induces the following commutative diagram with exact rows. 

\begin{equation} \label{PGO+LES}
\xymatrix{
1 \ar[r] &H^1(k,GO^+(A, \sigma)) \ar[r]^-{\eta} \ar[d]^-f &H^1(k,PGO^+(A, \sigma)) \ar[r]^-{\delta} \ar[d]^-g &H^2(k,G_m) \ar[d]^-h\\
1 \ar[r] & H^1(L,GO^+(A, \sigma)) \ar[r]^-{\eta} & H^1(L,PGO^+(A, \sigma)) \ar[r] & H^2(L, G_m)
}
\end{equation}

$H^1(k,PGO^+(A, \sigma))$ classifies $k$-isomorphism classes of triples $(A^\prime, \sigma^\prime, \phi^\prime)$ where $A^\prime$ is a central simple algebra over $k$ of the same degree as $A$, $\sigma^\prime$ is an orthogonal involution on $A^\prime$ and $\phi^\prime$ is an isomorphism from the center of the Clifford algebra of $A^\prime$ to the center of the Clifford algebra of $A$. For any such triple $(A^\prime, \sigma^\prime, \phi^\prime)$, $\delta(A^\prime, \sigma^\prime, \phi^\prime) = [A^\prime][A]^{-1}$ which is 2-torsion in the Brauer group since both $A$ and $A^\prime$ admit involutions of the first kind. Then, since $[L:k]$ is odd, $h$ is injective on the image of $\delta$ in $H^2(k,G_m)$. 

So choose $(A^\prime, \sigma^\prime, \phi^\prime) \in \ker(g)$. By commutativity of the rightmost rectangle in \eqref{PGO+LES}, $h(\delta(A^\prime, \sigma^\prime, \phi^\prime)) = \point$ and thus $\delta(A^\prime, \sigma^\prime, \phi^\prime) = \point$. Then by the exactness of the top row of the diagram, there is a $\lambda^{\prime} \in H^1(k, GO^+(A, \sigma))$ such that $\eta(\lambda^{\prime}) = (A^\prime, \sigma^\prime, \phi^\prime)$. By commutativity of the left rectangle of \eqref{PGO+LES}, $\eta(f(\lambda^\prime)) = \point$ which by exactness of the bottom row, gives $f(\lambda^\prime) = \point$. Then by \ref{GO+}, $\lambda^{\prime}=\point$ and thus $(A^\prime, \sigma^\prime, \phi^\prime) = \eta(\lambda^\prime) =\point$. 

\end{proof}

\section{Exceptional Groups} \label{EX}

The main result of this section is the following.

\begin{theorem}\label{exceptionalmeta}
Let $k$ be a field of characteristic different from 2. Let $G$ be a quasisplit, absolutely simple exceptional algebraic group over $k$ which is either simply connected or adjoint and is not of type $E_8$. Let $\{L_i\}_{1 \leq i \leq m}$ be a set of finite field extensions of k such that $\gcd([L_i:k])=d$. If $d$ is coprime to $S(G)$, then the canonical map
\[
H^1(k,G) \rightarrow \displaystyle \prod_{i=1}^m H^1(L_i,G)
\]
has trivial kernel.
\end{theorem}

In what follows, we consider absolutely simple groups of type $G_2$; absolutely simple, split groups of type $F_4$; quasisplit, absolutely simple, simply connected groups of types $^{3,6}D_4, E_6, E_7$ and then quasisiplit, absolutely simple, adjoint groups of types $^{3,6}D_4, E_6, E_7$. 

In fact, our results also hold in other settings. The proof we give for the case $G$ simple, split of type $F_4$ requires only the weaker assumption that $G$ is a simple group of type $F_4$ such that the invariant $g_3(G)=0$ where $g_3$ is an invariant $H^1(k, G) \to H^3(k, \mathbb{Z}/(3))$ defined by Rost \cite{Rostmod3}. Similary, given \cite{Skip} one can prove a similar result for $G$ isotropic of inner type $E_6$ .

Results on the kernel of the Rost invariant will be an important tool in this section. For $G$ absolutely simple, simply connected, the \emph{Rost invariant} $R_G$ is an invariant of $G$ with values in $ H^3(*, \mathbb{Q}/\mathbb{Z}(2))$ ,\cite{SkipMerkSerre}. Notably, the Rost invariant generates the group of all normalized invariants of $G$ with values in $ H^3(*, \mathbb{Q}/\mathbb{Z}(2))$ \cite[Theorem 9.11]{SkipMerkSerre}. 

\subsection*{G absolutely simple of type $G_2$}

\begin{proposition}\label{G2}
Let $k$ be a field of characteristic different from 2. Let $G$ be an absolutely simple group of type $G_2$ over $k$. Let $L$ be any  finite field extension of $k$ of odd degree. Then the canonical map
\[
H^1(k,G) \rightarrow H^1(L,G)
\]
has trivial kernel.
\end{proposition}

\begin{proof}

An absolutely simple group of type $G_2$ is isomorphic to a group of the form $Aut(C)$ where $C$ is a Cayley algebra over $k$. Since every Cayley algebra over $k$ splits over the algebraic closure of $k$, $H^1(k, G)$ is in bijection with the isomorphism classes of Cayley algebras over $k$ \cite[Proposition 29.1]{KMRT}. To any Cayley algebra $C^{\prime}$, we associate its norm form $q_{C^\prime}$ which is a 3-fold Pfister form. Choose $[C^{\prime}] \in ker(H^1(k,G) \rightarrow  H^1(L, G))$. Then $C^{\prime}_L$ is isomorphic to $C_L$ which is the point in $H^1(L, G)$. Since two isomorphic Cayley algebras have isometric norm forms \cite[Theorem 33.19]{KMRT}, $q_{C_L}$ is isometric to $q_{C^{\prime}_L}$. But since $[L:k]$ is odd, by Springer's theorem \cite[Chapter VII, Theorem 2.7]{Lam} ${q_C}$ is isometric to $q_{C^{\prime}}$. Applying \cite[Theorem 33.19]{KMRT} again, we conclude that $C^\prime$ is isomorphic to $C$,and $[C^\prime]$ is the point in $H^1(k,G)$.

\end{proof}

\subsection*{G absolutely simple, split of type $F_4$}

\begin{proposition} \label{F4Dynkin}
Let $k$ be a field of characteristic different from 2. Let $G$ be a split, absolutely simple, simply connected group of type $F_4$. Let $\{L_i\}_{1 \leq i \leq m}$ be a set of finite extensions of $k$ and let $\gcd([L_i:k])=d$. If $d$ is coprime to 2 and 3, then the canonical map
\[
H^1(k,G) \rightarrow \displaystyle \prod_{i=1}^m H^1(L_i,G)
\]
has trivial kernel.
\end{proposition}

\begin{proof}
We may regard the Rost invariant as a map $H^1(k,G) \to H^3(k, \mathbb{Z}/6\mathbb{Z}(2))$ \cite[pg 436]{KMRT}. The following diagram commutes

\[
\xymatrix{
H^1(k, G) \ar[r]^-{\psi} \ar[d]^-f & H^3(k, \mathbb{Z}/6\mathbb{Z}(2)) \ar[d]^-g\\
\prod H^1(L_i, G) \ar[r]^-{\psi_{L}} & \prod H^3(L_i, \mathbb{Z}/6\mathbb{Z}(2))
}
\]

Choose $J$ in $\ker(f)$. By commutativity of the diagram, $\psi(J)$ is in $\ker(g)$ and by our assumption on $d$, $g$ has trivial kernel. Thus $\psi(J) =0$ and $J$ is reduced \cite{Rostmod3}. Following Serre \cite[\S 9]{Serre} we associate to $J$ a quadratic form $\qform{2, 2, 2} \perp \phi_5 \perp -\phi_3$ which we denote $T_J$ and refer to as the trace form of $J$. Here each $\phi_n$ is an $n$-fold Pfister form. Since $J$ is reduced, it is determined up to isomorphism by the isometry class of $T_J$ \cite[\S 9]{Serre} which by Witt's cancellation is in turn determined by the isometry classes of $\phi_3$ and $\phi_5$. Since by assumption $J$ is split over each $L_i$, $\phi_3$ and $\phi_5$ are hyperbolic over each $L_i$. Since at least one of the $L_i$ is odd degree, by Springer's theorem \cite[Chapter VII, Theorem 2.7]{Lam} $\phi_3$ and $\phi_5$ are hyperbolic over $k$ from whence we have $J$ split over $k$. Thus $J$ is the point in $H^1(k, G)$.

\end{proof}

\subsection*{G quasisplit, absolutely simple, simply connected of type ${}^{3,6}D_4,E_6,E_7$}

\begin{proposition} \label{DE}
Let $k$ be a field and $G$ a quasisplit, absolutely simple, simply connected group of type ${}^{3,6}D_4,E_6,E_7$. Let $\{L_i\}_{1 \leq i \leq m}$ be a set of finite field extensions of $k$ and let $\gcd([L_i:k])=d$. If $d$ is coprime to 2 and 3 then the canonical map
\[
H^1(k,G) \rightarrow \displaystyle \prod_{i=1}^m H^1(L_i,G)
\]
has trivial kernel.
\end{proposition}

\begin{proof}
For $G$ absolutely simple, simply connected, the Rost invariant $R_G$ takes values in $(\mathbb{Z}/n_G\mathbb{Z})(2)$ where $n_G$ is the Dynkin index of $G$ \cite[pg 436]{KMRT}.

The following diagram commutes. 
\begin{equation}\label{rost}
\xymatrix{
H^1(k,G) \ar[r]^-{R_G} \ar[d]^-f &H^3(k, (\mathbb{Z}/n_G\mathbb{Z})(2)) \ar[d]^-g \\
\displaystyle \prod H^1(L_i, G) \ar[r]^-{R_G} &\displaystyle \prod H^3(L_i, (\mathbb{Z}/n_G\mathbb{Z})(2))}
\end{equation}

Choose $\lambda \in \ker(f)$. Let $R_G(\lambda) = \lambda^{\prime}$. Since $S(G)$ contains the prime divisors of $n_G$, and we have assumed $d$ coprime to $S(G)$, a restriction-corestriction argument gives that   $g$ has trivial kernel. So by commutativity of \eqref{rost}, $\lambda$ is in $\text{ker}(R_G)$. By a theorem of Garibaldi, \cite{Skip} $R_G$ has trivial kernel for $G$ of type $E_6, E_7$.That $R_G$ has trivial kernel for $G$ of type $^{3,6} D_4$ follows from results in \cite{KMRT}. Garibaldi has provided the details of the argument \cite{Skip2}. So we conclude that $\lambda$ = point.

\end{proof}

\subsection*{G quasisplit, absolutely simple, adjoint of type ${}^{3,6}D_4,E_6,E_7$}

The proof in this case is a consequence of the result in the simply connected case \ref{DE}. 

\begin{proposition} \label{DEad}
Let $k$ be a field of characteristic different from 2 and 3 and let $G$ be a quasisplit, absolutely simple, adjoint group of type ${}^{3,6}D_4, E_6, E_7$. Let $\{L_i\}$ be a set of finite field extensions of $k$ and let $\gcd([L_i:k])=d$. If $d$ is coprime to 2 and 3, then the canonical map
\[
H^1(k,G) \rightarrow \displaystyle \prod_{i=1}^m H^1(L_i,G)
\]
has trivial kernel.
\end{proposition}

\begin{proof}
We have a short exact sequence

\begin{equation} \label{DEASES}
\xymatrix{
1 \ar[r] &\mu \ar[r] &G^{sc} \ar[r]^-\pi &G \ar[r] &1}
\end{equation}

\noindent
where $G^{sc}$ is a simply connected cover of $G$ and $\mu$ is its center. Since $G$ is by assumption quasisplit, then $G^{sc}$ is quasisplit. So let $T$ be the maximal, quasitrivial torus in $G^{sc}$. 

As $\mu \subset T \subset G^{sc}$, the map $H^1(k, \mu) \to H^1(k, G^{sc})$ induced by the inclusion of $\mu$ in $G^{sc}$ factors through the map $H^1(k,\mu) \to H^1(k,T)$ induced by the inclusion of $\mu$ in $T$. But since $T$ is quasitrivial, $H^1(k,T)$ is trivial, and thus the image of the map $H^1(k, \mu) \to H^1(k,G^{sc})$ is trivial. Given this result,  \eqref{DEASES} induces the following commutative diagram with exact rows.

\begin{equation} \label{DEALES}
\xymatrix{
1 \ar[r] &H^1(k, G^{sc}) \ar[r]^-{\pi} \ar[d]^-f & H^1(k, G) \ar[d]^-g \ar[r]^-\delta & H^2(k, \mu) \ar[d]^-h\\
1 \ar[r] & \prod H^1(L_i, G^{sc}) \ar[r]^-\pi &\prod H^1(L_i, G) \ar[r]^-\delta &\prod H^2(L_i, \mu)
}
\end{equation}

Choose $\lambda \in \text{ker}(g)$. The prime divisors of the order of  $\mu$ are contained in $S(G)$. Then since $d$ is coprime to $S(G)$, $d$ is coprime to the order of $\mu$ and a restriction-corestriction argument gives that $h$ has trivial kernel. So by commutativity of the rightmost rectangle of \eqref{DEALES}, $\lambda \in \text{ker}(\delta)$. By exactness of the top row of \eqref{DEALES} choose $\lambda^{\prime} \in H^1(k, G^{sc})$ such that $\pi(\lambda^\prime) = \lambda$. Commutativity of the left rectangle of \eqref{DEALES} gives $f(\lambda^\prime) \in \text{ker}(\pi)$ which is trivial by the exactness of the bottom row of \eqref{DEALES}. So $f(\lambda^\prime) =$ point, from whence by \ref{DE}, $\lambda^{\prime}$ is the point in $H^1(k,G^{sc})$. It is then immediate that $\lambda = \pi(\lambda^{\prime})$ is the the point in $H^1(k, G)$.
\end{proof}

\begin{remark}
One can avoid restrictions on the characteristic $k$ in \ref{DEad} by giving a proof in the flat cohomology sets $H^1_{\fppf}(*,*)$ as defined in \cite{Waterhouse}. Since $G$ is by assumption smooth, $H^1_{\fppf}(k,G)= H^1(k,G)$. 
\end{remark}

\section{Main Result} \label{ME}

\begin{theorem} \label{bigmeta}
Let $k$ be a field of characteristic different from 2. Let $\{L_i\}_{1 \leq i \leq m}$ be a set of finite field extensions of $k$ and let $\gcd([L_i:k])=d$. Let $G$ be a simply connected or adjoint semisimple algebraic $k$-group which does not contain a simple factor of type $E_8$ and such that every exceptional simple factor of type other than $G_2$ is quasisplit. If $d$ is coprime to $S(G)$, then the canonical map
\[
H^1(k,G) \to \prod_{i=1}^m H^1(L_i,G)
\]
has trivial kernel.
\end{theorem}

\begin{proof}

Write $G$ as a product of groups of the form $R_{E_j/k}G_j$ where each $G_j$ is an absolutely simple, simply connected or adjoint group and each $E_j$ is a finite separable field extension of $k$. It is sufficient to consider a group of the form $R_{E/k}G$ for an absolutely simple group $G$ and a finite separable field extension $E$ of $k$. By Shapiro's Lemma, $H^1(k, R_{E/k}G) \cong H^1(E,G)$ and $\prod_i H^1(L_i, R_{E/k}G) \cong \prod_i H^1_{\et}(L_i \otimes E, G)$ where the subscript $\et$ denotes the \'etale cohomology as in \cite{Milne}. 

Since $E$ is separable, for each index $i$, $E \otimes L_i \cong \prod_s L_{i,s}$ for $L_{i, s}$ finite extensions of $E$ and therefore $H^1_{\et}(L_i \otimes E, G) \cong \prod_s H^1(L_{i,s}, G)$. Let $d^\prime$ be the greatest common divisor of $\sum_{i,s}[L_{i,s}:k]$. Since for each $i$, $\sum_s[L_{i,s}:k]=[L_i:k]$, then $d^\prime$ divides $d$ and thus $d^\prime$ is coprime to $S(G)$. Thus the map $H^1(E,G) \to \prod_i \prod_s H^1_{\et}(L_{i,s}, G)$ has trivial kernel in view of \ref{scmeta}, \ref{admeta} and \ref{exceptionalmeta} above and thus the map $H^1(k,G) \to \prod_{i=1}^m H^1(L_i,G)$ has trivial kernel. 

\end{proof}

\bibliographystyle{amsplain}
\bibliography{paper1d}
\end{document}